\documentclass[10pt]{amsart}

\usepackage{amsmath,amssymb,amsfonts,amsthm,amscd}
\usepackage{pstricks}

\theoremstyle{plain}
\newtheorem{thm}[subsection]{Theorem}
\newtheorem{lem}[subsection]{Lemma}
\newtheorem{prop}[subsection]{Proposition}
\newtheorem{cor}[subsection]{Corollary}
\newtheorem{ex}[subsection]{Example}
\newtheorem{rem}[subsection]{Remark}

\newtheorem{defn}[subsection]{Definition}

\newcommand{\ga}{\alpha}
\newcommand{\gb}{\beta}

\newcommand{\gl}{\lambda}

\newcommand{\gm}{\mu}
\newcommand{\gn}{\nu}
\newcommand{\gx}{\xi}
\newcommand{\go}{\omega}

\newcommand{\gr}{\rho}
\newcommand{\gs}{\sigma}

\newcommand{\cB}{\mathcal{B}}

\newcommand{\cL}{\mathcal{L}}
\newcommand{\cP}{\mathcal{P}}
\newcommand{\cR}{\mathcal{R}}

\newcommand{\cT}{\mathcal{T}}

\newcommand{\bZ}{\mathbb{Z}}
\newcommand{\bN}{\mathbb{N}}
\newcommand{\bC}{\mathbb{C}}

\newcommand{\bs}{\boldsymbol}
\newcommand{\bgl}{\bs{\gl}}
\newcommand{\blm}{_{\bgl}^\gm}
\newcommand{\Bl}{\cB_{\bgl}}
\newcommand{\Llm}{\cL\cR_{\bgl}^\gm}
\newcommand{\by}{\mathbf{y}}

\newcommand{\yo}{y^{(1)}}
\newcommand{\yt}{y^{(2)}}
\newcommand{\yth}{y^{(3)}}
\newcommand{\yi}{y^{(i)}}

\newcommand{\yr}{y^{(r)}}

\newcommand{\xy}{(x\,|\,y)}
\newcommand{\xyo}{(x\,|\,\yo)}

\newcommand{\xyi}{(x\,|\,\yi)}

\newcommand{\xyr}{(x\,|\,\yr)}
\newcommand{\xyb}{(x\,|\,\by)}

\newcommand{\hn}{_{(n)}}

\def\ni{\noindent}
\def\DS{\displaystyle}

\newcommand{\ol}{\overline}

\newcommand{\sgn}{\mathop{\rm sgn}\nolimits}
\newcommand{\ch}{\mathop{\rm Char}\nolimits}
\newcommand{\vs}{\vspace{1em}}

\newrgbcolor{ncyn}{0 1 1}

\title{Products of Factorial Schur Functions}
\author{Victor Kreiman}
\address{Department of Mathematics, University of Georgia\\
Athens, GA 30602 USA}
\email{vkreiman@math.uga.edu}

\begin{document}
\maketitle

\begin{abstract}
The product of any finite number of factorial Schur functions can
be expanded as a $\bZ[\by]$-linear combination of Schur functions.
We give a rule for computing the coefficients in such an expansion
which generalizes a specialization of the Molev-Sagan rule, which
in turn generalizes the classical Littlewood-Richardson rule.
\end{abstract}

\setcounter{tocdepth}{1} \tableofcontents

\section{Introduction}

Let $\cP_n$ denote the set of partitions
$\{\gl=(\gl_1,\ldots,\gl_n)\in\bN^n\mid\gl_1\geq\cdots\geq
\gl_n\}$, and for $\gl\in\cP_n$, let $\cT_{\gl}$ denote the set of
all Young tableaux of shape $\gl$ with entries in
$\{1,\ldots,n\}$.  Let $x=(x_1,\ldots,x_n)$ and
$y=(y_1,y_2,\ldots)$ be two sets of variables.  The Schur function
$s_\gl(x)$ and factorial Schur function $s_\gl\xy$ are defined as
follows:
\begin{equation*}
s_\gl(x)=\sum_{T\in \cT_{\gl}}\prod_{a\in T} x_{a}\quad\text{and}\quad
s_{\gl}\xy=\sum_{T\in \cT_{\gl}}\prod_{a\in T}\left(x_{a}+y_{a+c(a)-r(a)}\right),
\end{equation*}
where for entry $a\in T$, $c(a)$ and $r(a)$ are the column and row
numbers of $a$ respectively.  Note that if $y$ is  specialized to
$(0,0,\ldots)$, then $s_\gl\xy=s_\gl(x)$. Factorial Schur
functions are special cases of Lascoux and Sch\"utzenberger's
double Schubert polynomials \cite{La-Sc1, La-Sc2}. Various
versions of factorial Schur functions and their properties have
been studied by \cite{Bi-Lo}, \cite{Ch-Lo}, \cite{Go-Gr},
\cite{La}, \cite{Mac1}, \cite{Mac2}, \cite{Mo1}, and \cite{Mo2}
(see \cite{Mi} and \cite{Mo-Sa} for more discussion of these
functions).

Let $r\in\bN$.  For $i\in\{1,\ldots,r\}$, let
$y^{(i)}=(y^{(i)}_1,y^{(i)}_2,\ldots)$ be an infinite set of
variables, and let $\by$ denote the set of variables
$(y^{(1)},\dots, y^{(r)})$. As $\gl$ varies over $\cP_n$, both the
Schur functions $s_\gl(x)$ and the factorial Schur functions
$s_\gl\xyi$ (where $i\in\{1,\ldots,r\}$ is fixed) form
$\bZ[\by]$-bases for $\bZ[x,\by]^{S_n}$, the ring of polynomials
in the $x$ and $\by$ variables which are symmetric in the $x$
variables. Hence for any sequence
$\bgl=(\gl^{(1)},\ldots,\gl^{(r)})$ of elements of $\cP_n$, the
polynomial $s_{\bgl}\xyb=s_{\gl^{(1)}}\xyo\cdots
s_{\gl^{(r)}}\xyr$ can be expanded as a $\bZ[\by]$-linear
combination of Schur functions:
\begin{equation}\label{e.lr_def}
s_{\bgl}\xyb=\sum_{\gm\in\cP_n} c\blm(\by) s_\gm(x),\quad\text{for some } c\blm(\by)\in \bZ[\by].
\end{equation}
We give a rule for computing the coefficients $c\blm(\by)$ for
arbitrary $\bgl$. Our rule generalizes a rule by Molev-Sagan
\cite[Theorem 3.1]{Mo-Sa}, which computes $c\blm(\by)$ for $r=2$,
$\yo=(0,0,\ldots)$. The Molev-Sagan rule in turn generalizes the
classical Littlewood-Richardson rule \cite{Li-Ri}, which computes
$c\blm(\by)$ for $r=2$, $\yo=\yt=(0,0,\ldots)$. When $r=1$, the
coefficients $c\blm(\by)$ give the change of $\bZ[\by]$-basis
coefficients between the Schur and factorial Schur functions.
Another rule for these change of basis coefficients which, like
ours, is tableau-based, is given by Molev \cite{Mo1}, and a
determinantal formula is given by Macdonald \cite{Mac2}. More
generally, change of basis coefficients between Schubert and
double Schubert polynomials were obtained by Macdonald \cite{Mac3}
and Lascoux \cite{La2}.

A related problem is to expand $s_{\bgl}\xyb$ in the basis of
factorial Schur functions $\{s_\gm\xyo, \gm\in \cP_n\}$. A
solution to this problem for $r=2$ was given by Molev-Sagan
\cite[Theorem 3.1]{Mo-Sa} (the Molev-Sagan rule discussed above is
obtained from this one by specializing  $\yo=(0,0,\ldots)$), and a
solution for $r=2$, $\yo=\yt$ which is positive in the sense of
Graham \cite{Gr} was given by Knutson-Tao \cite{Kn-Ta}, Molev
\cite{Mo1}, and Kreiman \cite{Kr2}.

The proof of our rule for $c\blm(\by)$ generalizes a concise proof
by Stembridge \cite{Stem1} of the classical Littlewood-Richardson
rule. Stembridge's proof relies on sign-reversing involutions on
skew tableaux which were introduced by Bender and Knuth
\cite{Be-Kn}. We generalize these arguments and constructions to
barred skew tableaux, which are refinements of skew tableaux. Our
proof is similar to but simpler than the proof used in \cite{Kr2},
where Stembridge's methods are generalized to hatted skew
tableaux.


\textbf{Acknowledgements.} We thank W. Graham and A. Molev for
helpful comments and suggestions.

\section{Computing the Coefficients $c\blm(\by)$}

As above, let $\bgl=(\gl^{(1)},\ldots,\gl^{(r)})$ be a sequence of
elements in $\cP_n$, or, alternatively, of Young diagrams. Denote
also by $\bgl$ the \textbf{skew diagram} formed by placing each
Young diagram in the sequence $\gl^{(1)},\ldots,\gl^{(r)}$ below
and to the left of the preceding one. A \textbf{barred skew
tableau $T$ of shape $\bgl$} is a filling of the boxes of the skew
diagram $\bgl$ with elements of
$\{1,\ldots,n\}\cup\{\ol{1},\ldots,\ol{n}\}$ in such a way that
the entries weakly increase along any row from left to right and
strictly increase along any column from top to bottom, without
regard to whether or not the entries are barred.

The \textbf{unbarred column word of $T$} is the sequence of
unbarred entries of $T$ obtained by beginning at the top of the
rightmost column, reading down, then moving to the top of the next
to rightmost column and reading down, etc. (the barred entries are
just skipped over in this process). We say that the unbarred
column word of $T$ is \textbf{Yamanouchi} if, when one writes down
its entries and stops at any point, one will have written at least
as many ones as twos, at least as many twos as threes, etc. The
\textbf{unbarred content of $T$} is
$\go(T)=(\gx_1,\ldots,\gx_n)\in\bN^{n}$, where $\gx_k$ is the
number of unbarred $k$'s in $T$. Define $c_T(\by)=\prod
y^{(i(a))}_{a+c(a)-r(a)}$, where the product is over all barred
$a\in T$, $i(a)$ is the index of the particular Young diagram in
which $a$ resides, and $c(a)$ and $r(a)$ denote the column and row
numbers of $a$ in this Young diagram.

\begin{figure}[!h]
\psset{unit=.5cm,linewidth=.02}
\pspicture(0,0)(11,9)
%
%
\psset{linecolor=ncyn} \psframe*(0,3)(1,4) \psframe*(1,1)(2,3)
\psframe*(3,2)(4,3) \psset{linecolor=black}
\psline(0,4)(5,4) \psline(0,3)(5,3) \psline(0,2)(4,2)
\psline(0,1)(2,1) \psline(0,0)(1,0)
\psline(0,0)(0,4) \psline(1,0)(1,4) \psline(2,1)(2,4)
\psline(3,2)(3,4) \psline(4,2)(4,4) \psline(5,3)(5,4)
\rput{0}(.5,.5){$5$}
\rput{0}(.5,1.5){$4$} \rput{0}(1.5,1.5){$\ol{5}$}
\rput{0}(.5,2.5){$3$} \rput{0}(1.5,2.5){$\ol{3}$}
\rput{0}(2.5,2.5){$4$} \rput{0}(3.5,2.5){$\ol{5}$}
\rput{0}(.5,3.5){$\ol{1}$} \rput{0}(1.5,3.5){$1$}
\rput{0}(2.5,3.5){$1$} \rput{0}(3.5,3.5){$1$}
\rput{0}(4.5,3.5){$2$}
%
%
\psset{linecolor=ncyn}\psframe*(5,5)(7,6)
\psset{linecolor=black}
\psline(5,6)(7,6) \psline(5,5)(7,5) \psline(5,4)(6,4)
\psline(5,4)(5,6) \psline(6,4)(6,6) \psline(7,5)(7,6)
\rput{0}(5.5,4.5){$4$}
\rput{0}(5.5,5.5){$\ol{3}$} \rput{0}(6.5,5.5){$\ol{4}$}
%
%
\psset{linecolor=ncyn}\psframe*(9,8)(10,9)
\psset{linecolor=black}
\psline(7,6)(9,6) \psline(7,7)(9,7) \psline(7,8)(11,8)
\psline(7,9)(11,9)
\psline(7,6)(7,9) \psline(8,6)(8,9) \psline(9,6)(9,9)
\psline(10,8)(10,9) \psline(11,8)(11,9)
\rput{0}(7.5,6.5){$3$} \rput{0}(8.5,6.5){$3$}
\rput{0}(7.5,7.5){$2$} \rput{0}(8.5,7.5){$2$}
\rput{0}(7.5,8.5){$1$} \rput{0}(8.5,8.5){$1$}
\rput{0}(9.5,8.5){$\ol{1}$} \rput{0}(10.5,8.5){$1$}
\endpspicture
\caption{\label{f.tableau}
A barred skew tableaux $T$ of shape $\bgl$ and unbarred
content $\gm$, where $\bgl=((4,2,2),(2,1),(5,4,2,1))$,
and $\gm=\go(T)=(6,3,3,3,1)$. The unbarred column word of $T$ is
$1123123421141345$, which is Yamanouchi.  Thus $T\in\Llm$.
We have
$c_T(\by)=\prod_{a\in T,a\text{ barred}}
  y^{(i(a))}_{a+c(a)-r(a)}
 =\yo_{1+3-1}\,\yt_{3+1-1}\,\yt_{4+2-1}\,\yth_{1+1-1}\,
   \yth_{3+2-2}\,\yth_{5+4-2}\,\yth_{5+2-3}=\yo_3\yt_3\yt_5$
$\yth_1\yth_3\yth_7\yth_4$.
}
\end{figure}

\begin{defn}
Denote the set of all barred skew tableaux of shape $\bgl$ by
$\Bl$, and the set of all barred skew tableaux of $\Bl$ of
unbarred content $\gm$ whose unbarred column words are Yamanouchi
by $\Llm$.
\end{defn}
\begin{thm}\label{t.lr_rule}
$\displaystyle
c\blm(\by)=\sum_{T\in\Llm}c_T(\by)=\sum_{T\in\Llm}\left(
\prod_{a\in T\atop a\text{
barred}}y^{(i(a))}_{a+c(a)-r(a)}\right)$.
\end{thm}

\begin{ex}\label{ex.lr_rule} Let $n=2$, $\bgl=((2,1),(1,1))$,
and $\gm=(2,2)$. We list all $T\in\Llm$, and for each $T$ we
give $c_T(\by)$:\\[.2em]
\begin{center}
$
\begin{array}{c@{\hspace{3.5em}}c@{\hspace{3.5em}}c@{\hspace{3.5em}}c}
\psset{unit=.5cm,linewidth=.02} \pspicture(0,-1.5)(3,4)
\psline(0,2)(1,2) \psline(0,1)(1,1) \psline(0,0)(1,0)
\psline(0,0)(0,2) \psline(1,0)(1,2)
\rput{0}(.5,.5){$2$}
\rput{0}(.5,1.5){$1$}

\psset{linecolor=ncyn} \psframe*(1,3)(2,4) \psset{linecolor=black}
\psline(1,4)(3,4) \psline(1,3)(3,3) \psline(1,2)(2,2)
\psline(1,2)(1,4) \psline(2,2)(2,4) \psline(3,3)(3,4)
\rput{0}(1.5,2.5){$2$} \rput{0}(1.5,3.5){$\ol{1}$}
\rput{0}(2.5,3.5){$1$}
\rput{0}(1.5,-.8){$c_T(\by)=\yo_1$}
\endpspicture
&

\psset{unit=.5cm,linewidth=.02} \pspicture(0,-1.5)(3,4)
\psline(0,2)(1,2) \psline(0,1)(1,1) \psline(0,0)(1,0)
\psline(0,0)(0,2) \psline(1,0)(1,2)
\rput{0}(.5,.5){$2$}
\rput{0}(.5,1.5){$1$}

\psset{linecolor=ncyn} \psframe*(2,3)(3,4) \psset{linecolor=black}
\psline(1,4)(3,4) \psline(1,3)(3,3) \psline(1,2)(2,2)
\psline(1,2)(1,4) \psline(2,2)(2,4) \psline(3,3)(3,4)
\rput{0}(1.5,2.5){$2$} \rput{0}(1.5,3.5){$1$}
\rput{0}(2.5,3.5){$\ol{1}$}
\rput{0}(1.5,-.8){$c_T(\by)=\yo_2$}
\endpspicture
&

\psset{unit=.5cm,linewidth=.02} \pspicture(0,-1.5)(3,4)
\psline(0,2)(1,2) \psline(0,1)(1,1) \psline(0,0)(1,0)
\psline(0,0)(0,2) \psline(1,0)(1,2)
\rput{0}(.5,.5){$2$}
\rput{0}(.5,1.5){$1$}

\psset{linecolor=ncyn} \psframe*(2,3)(3,4) \psset{linecolor=black}
\psline(1,4)(3,4) \psline(1,3)(3,3) \psline(1,2)(2,2)
\psline(1,2)(1,4) \psline(2,2)(2,4) \psline(3,3)(3,4)
\rput{0}(1.5,2.5){$2$} \rput{0}(1.5,3.5){$1$}
\rput{0}(2.5,3.5){$\ol{2}$}
\rput{0}(1.5,-.8){$c_T(\by)=\yo_3$}
\endpspicture
&

\psset{unit=.5cm,linewidth=.02} \pspicture(0,-1.5)(3,4)
\psset{linecolor=ncyn} \psframe*(0,1)(1,2) \psset{linecolor=black}
\psline(0,2)(1,2) \psline(0,1)(1,1) \psline(0,0)(1,0)
\psline(0,0)(0,2) \psline(1,0)(1,2)
\rput{0}(.5,.5){$2$}
\rput{0}(.5,1.5){$\ol{1}$}

\psline(1,4)(3,4) \psline(1,3)(3,3) \psline(1,2)(2,2)
\psline(1,2)(1,4) \psline(2,2)(2,4) \psline(3,3)(3,4)
\rput{0}(1.5,2.5){$2$} \rput{0}(1.5,3.5){$1$}
\rput{0}(2.5,3.5){$1$}
\rput{0}(1.5,-.8){$c_T(\by)=\yt_1$}
\endpspicture
\end{array}
$
\end{center}
By Theorem \ref{t.lr_rule}, $c\blm(\by)=\yo_1+\yo_2+\yo_3+\yt_1$.
\end{ex}

\begin{rem}\label{rem.specialize_to_schur}
If $\yi$ is specialized to $(0,0,\ldots)$ for some $i$, then
$s_{\gl^{(i)}}\xyi=s_{\gl^{(i)}}(x)$. In this case, if $T\in\Llm$
has a barred entry in Young diagram $\gl^{(i)}$, then
$c_T(\by)=0$. Thus in Theorem \ref{t.lr_rule}, $c\blm(\by)$ is
computed by summing over only those $T\in\Llm$ with no barred
entries in Young diagram $\gl^{(i)}$.
\end{rem}

\begin{rem} For simplicity, we assume here that $y:=\yo=\cdots=\yr$,
and we denote $\by$ by just $y$. Theorem \ref{t.lr_rule} has the
following representation-theoretic interpretation. Let
$G=GL_n(\bC)$, and let $H=\bC^*\times\bC^*\times\cdots$. For
$\gl\in\cP_n$, let $V_\gl$ denote the irreducible $G$
representation of highest weight $\gl$. For
$m\in\bZ[y]=\bZ[y_1,y_2,\ldots]$ a monomial with coefficient $1$,
let $L_m$ denote the $H$ representation with character $m$. For
$\gl\in\cP_n$, define the following $G\times H$ representation:
\begin{equation*}
U_\gl=\bigoplus_{\gm\in\cP_n\atop T\in\cL\cR_{(\gl)}^\gm}  V_\gm\otimes L_{c_T(y)}.
\end{equation*}
By Theorem \ref{t.lr_rule}, $\ch_{G\times
H}U_\gl=\sum_{\gm\in\cP_n,\, T\in\cL_{(\gl)}^\gm}
s_\gm(x)c_T(y)=s_\gm\xy$.

Let $R(G\times H)$ denote the polynomial representation ring of
$G\times H$ (i.e., the subring of the full representation ring of
$G\times H$ generated by the polynomial representations) and
$R(H)$ the polynomial representation ring of $H$.  Then $R(G\times
H)\cong \bZ[x]^{S_n}\otimes\bZ[y]=\bZ[x,y]^{S_n}$, and
$R(H)\cong\bZ[y]$. Since $\{s_\gl\xy\mid\gl\in\cP_n\}$ forms a
$\bZ[y]$-basis for $\bZ[x,y]^{S_n}$, the classes
$\{[U_\gl]\mid\gl\in\cP_n\}\subset R(G\times H)$ form an
$R(H)$-basis for $R(G\times H)$. Theorem \ref{t.lr_rule} implies
the following decomposition as $G\times H$ representations:
\begin{equation*}
U_{\gl^{(1)}}\otimes\cdots\otimes U_{\gl^{(t)}}=
\bigoplus_{\gm\in\cP_n\atop T\in\cL\cR\blm} V_\gm\otimes L_{c_T(y)}.
\end{equation*}
\end{rem}

\begin{rem}
Factorial Schur functions can be defined in terms of reverse Young
tableaux instead of Young tableaux (see \cite{Kr2}).  Beginning
with this definition, and with several minor adjustments to the
proofs, one can obtain a rule for $c\blm(\by)$ which is almost the
same as that of Theorem \ref{t.lr_rule}, except that it is
expressed in terms of reverse barred skew tableaux instead of
barred skew tableaux, with appropriate adjustments to the indexing
of columns and rows. This rule is similar in form to the factorial
Littlewood-Richardson rule of Molev \cite{Mo1} and Kreiman
\cite{Kr2}, which is equivalent to the Knutson-Tao rule
\cite{Kn-Ta}.
\end{rem}

\section{Generalization of Stembridge's Proof}

In this section we prove Theorem \ref{t.lr_rule}. The underlying
structure and logic of our proof follows Stembridge \cite{Stem1}.
Our approach is to generalize Stembridge's methods
from skew tableaux to barred skew tableaux.

For $\gx=(\gx_1,\ldots,\gx_n)\in\bN^n$, define
$a_\gx(x)=\det[(x_i)^{\gx_j}]_{1\leq i,j\leq n}$. Let
$\gr=(n-1,n-2,\ldots,0)\in\cP_n$.
\begin{lem}\label{l.lr}  $a_{\gr}(x)s_{\bgl}\xyb
=\sum_{T\in\Bl}c_T(\by) a_{\gr+\go(T)}(x)$.
\end{lem}

\begin{lem}\label{l.bad_guys_vanish} $\sum c_T(\by)
a_{\gr+\go(T)}(x)=0$, where the sum is over all $T\in\Bl$ such
that the unbarred column word of $T$ is not Yamanouchi.
\end{lem}

\ni The proofs of these two lemmas appear at the end of this
section.  The following three corollaries are easy consequences.

\begin{cor}\label{c.lr1}
$a_{\gr}(x)s_{\bgl}\xyb =\sum c_T(\by) a_{\gr+\go(T)}(x)$, where
the sum is over all $T\in\Bl$ such that the unbarred column word
of $T$ is Yamanouchi.
\end{cor}
\ni If we set $r=1$, $\bgl=(\gl)$, and $\by=((0,0,\ldots))$ in
Corollary \ref{c.lr1}, then $s_{\bgl}\xyb=s_\gl(x)$, and there is
a unique $T\in\Bl$ with no barred entries (so that $c_T(\by)\neq
0$; see Remark \ref{rem.specialize_to_schur}) whose unbarred
column word is Yamanouchi: namely, the tableau whose $i$-th row
contains all $i$'s.  For this $T$, $c_T(\by)=1$ and $\go(T)=\gl$.
Thus we obtain the well known bialternant formula for the Schur
function.
\begin{cor}\label{c.lr2}
$\DS s_\gl(x)=a_{\gr+\gl}(x)/a_\gr(x)$.
\end{cor}
\ni Dividing both sides of Corollary \ref{c.lr1} by $a_\gr(x)$ and
applying Corollary \ref{c.lr2} yields
\begin{cor}\label{c.lr3}
$s_{\bgl}\xyb =\sum c_T(\by) s_{\go(T)}(x)$, where the sum is over
all $T\in\Bl$ such that the unbarred column word of $T$ is
Yamanouchi.
\end{cor}
\ni Regrouping terms in the summation,
$ s_{\bgl}\xyb=\sum_{\gm\in\cP_n}\sum_{T\in\Llm} c_T(\by)
s_\gm(x)$. This proves Theorem \ref{t.lr_rule}.

\subsection*{Involutions on Barred Skew Tableaux}

The proofs of Lemmas \ref{l.lr} and \ref{l.bad_guys_vanish} rely
on involutions $s_1,\ldots,s_{n-1}$ on $\Bl$, which generalize the
involutions on Young tableaux introduced by Bender and Knuth
\cite{Be-Kn}. We now define these involutions and prove several of
their properties.

Let $T\in\Bl$, and let $i\in\{1,\ldots,n-1\}$ be fixed.  Let $a$
be an entry of $T$. Then either $a=j$ or $a=\ol{j}$, where
$j\in\{1,\ldots,n\}$. We call $j$ the \textbf{value} of $a$. We
say that an entry of $T$ of value $i$ or $i+1$ is \textbf{free} if
there is no entry of value $i+1$ or $i$ respectively in the same
column; \textbf{semi-free} if there is an entry of value $i+1$ or
$i$ respectively in the same column, and exactly one of the two is
barred; or \textbf{locked} if there is an entry of value $i+1$ or
$i$ respectively in the same column, and either both entries are
unbarred or both entries are barred. Note that any entry of value
$i$ or $i+1$ must be exactly one of these three types. In any row,
the free entries are consecutive. Semi-free entries come in pairs,
one below the other, as do locked entries.

The barred skew tableau $s_iT$ is obtained by applying
$\textbf{Alorithm 1}$ to each semi-free pair of entries in $T$,
and applying $\textbf{Algorithm 2}$ to each maximal string of
consecutive free entries $S$ lying on the same row. \vs

\begin{quote}
\noindent\textbf{Algorithm 1}\\ {\it\ni For a semi-free pair
consisting of two entries lying in the same column of $T$, the bar
is removed from the barred entry and placed on top of the unbarred
entry.}
\end{quote}

\begin{quote}
\noindent\textbf{Algorithm 2}\\
{\it\ni Let $l$ be the number of (unbarred) $i$'s and $r$ the
number of
(unbarred) $i+1$'s that $S$ contains.
\begin{itemize}
\item If $l=r$: Do not change $S$.

\item If $l<r$: If $l=0$ then define $R=S$.  Otherwise,
    letting $S_1$ be the $l$-th $i$ of $S$ from the left and
    $S_2$ the $l$-th $i+1$ of $S$ from the right, define $R$
    to be the string of consecutive entries of $S$ beginning
    with the first $i+1$ or $\ol{i+1}$ of $S$ to the right of
    $S_1$ and ending with the first $i+1$ of $S$ to the left
    of $S_2$. Note that each entry of $R$ is either an $i+1$
    or $\ol{i+1}$. Modify $R$ as follows:
    \begin{itemize}
    \item[1.] change each $i+1$ to an $i$ and each
        $\ol{i+1}$ to an $\ol{i}$; and then

    \item[2.] beginning with the rightmost $\ol{i}$ and
        then proceeding to the left, swap each $\ol{i}$
        with the $i$ immediately to the right of it.
    \end{itemize}

\item If $l>r$: If $r=0$ then define $R=S$. Otherwise, letting
    $S_1$ be the $r$-th $i$ of $S$ from the left and $S_2$ the
    $r$-th $i+1$ of $S$ from the right, define $R$ to be the
    string of consecutive entries of $S$ beginning with the
    first $i$ of $S$ to the right of $S_1$ and ending with the
    first $i$ or $\ol{i}$ of $S$ to the left of $S_2$. Note
    that each entry of $R$ is either an $i$ or an $\ol{i}$.
    Modify $R$ as follows:
    \begin{itemize}
    \item[1.] change each $i$ to an $i+1$ and each
        $\ol{i}$ to an $\ol{i+1}$; and then

    \item[2.] beginning with the leftmost $\ol{i+1}$ and
        then proceeding to the right, swap each $\ol{i+1}$
        with the $i+1$ immediately to the left of it.
    \end{itemize}
\end{itemize}
}
\end{quote}
\vs

\ni One checks that $s_i$ is an involution on $\Bl$. Let $S_n$
denote the permutation group on $n$ elements and $\gs_i$ the
simple transposition of $S_n$ which exchanges $i$ and $i+1$. The
following Lemma follows from the construction of $s_i$.

\begin{lem}\label{l.transp_properties} Let $T\in\Bl$.  Then

\ni (i) $c_{s_iT}(\by)= c_T(\by)$.

\ni (ii) $\go(s_iT)=\gs_i \go(T)$.
\end{lem}

Let $T\in\Bl$ and let $\gs\in S_n$. Choose some decomposition of
$\gs$ into simple transpositions: $\gs=\gs_{i_1}\cdots\gs_{i_t}$,
and define $\gs T=s_{i_1}\cdots s_{i_t}T$. By Lemma
\ref{l.transp_properties},
\begin{equation}\label{e.perm_properties}c_{\gs T}(\by)=c_T(\by)\qquad\text{and}\qquad
\go(\gs T)=\gs\go(T).
\end{equation}
In particular, although $\gs T$ depends on the decomposition of
$\gs$, both $c_{\gs T}(\by)$ and $\go(\gs T)$ are independent of
the decomposition.

\subsection*{Proof of Lemmas \ref{l.lr} and \ref{l.bad_guys_vanish}}

\begin{proof}[Proof of Lemma \ref{l.lr}]
Expanding $s_{\bgl}\xyb$ into monomials:
\begin{align*}
s_{\bgl}\xyb&=\prod_{i=1}^r s_{\gl^{(i)}}\xyi\\
&=\prod_{i=1}^r \sum_{T\in \cT_{\gl^{(i)}}}
\prod_{a\in T}\left(x_{a}+y^{(i)}_{a+c(a)-r(a)}\right)\\
%
%
&=\prod_{i=1}^r \sum_{T\in \cB_{(\gl^{(i)})}}
\prod_{{a\in T\atop a\text{
unbarred}}}x_{a}\prod_{a\in T\atop a\text{
barred}}y^{(i)}_{a+c(a)-r(a)}\\
&=\sum_{T\in \cB_{\bgl}}
\prod_{{a\in T\atop a\text{
unbarred}}}x_{a}\prod_{a\in T\atop a\text{
barred}}y^{(i(a))}_{a+c(a)-r(a)}
=\sum_{T\in \Bl}\,x^{\go(T)}c_T(\by).
\intertext{Therefore}
a_{\gr}(x)s_{\bgl}\xyb&=
\sum_{\gs\in S_n}\sum_{T\in\Bl}\sgn(\gs)
x^{\gs(\gr)}x^{\go(T)}c_T(\by)\\
&=\sum_{\gs\in S_n}\sum_{T\in\Bl}\sgn(\gs)
x^{\gs(\gr)}x^{\go(\gs T)}c_{\gs T}(\by)\\
&=\sum_{\gs\in S_n}\sum_{T\in\Bl}c_T(\by)\sgn(\gs)
x^{\gs(\gr+\go(T))}
=\sum_{T\in\Bl}c_T(\by)a_{\gr+\go(T)}(x).
\end{align*}
\ni The second equality follows from the fact that $\gs$ is an
involution on $\Bl$; thus as $T$ varies over $\Bl$, so does $\gs
T$. The third equality follows from (\ref{e.perm_properties}).
\end{proof}

\begin{proof}[Proof of Lemma \ref{l.bad_guys_vanish}]
For $T\in\Bl$ and $j$ a positive integer, define $T_{<j}$ to be
the barred skew tableau consisting of the columns of $T$ lying to
the left of column $j$ (and similarly for $T_{\leq j}$, $T_{>j}$,
$T_{\geq j}$).

We will call the $T\in\Bl$ for which $\go(T_{\geq j})\not\in\cP_n$
for some $j$ {\it Bad Guys} (i.e., $T$ is a Bad Guy if and only if
its unbarred column word is not Yamanouchi). Let $T$ be a Bad Guy,
and let $j$ be maximal such that $\go(T_{\geq j})\not\in\cP_n$.
Having selected $j$, let $i$ be minimal such that $\go(T_{\geq
j})_{i}<\go(T_{\geq j})_{i+1}$. Since $\go(T_{>
j})_{i}\geq\go(T_{> j})_{i+1}$ (by the maximality of $j$), we must
have $\go(T_{> j})_{i}=\go(T_{> j})_{i+1}$, and column $j$ of $T$
must have an unbarred $i+1$ but not an unbarred $i$. Thus
\begin{equation}\label{e.tech}
(\gr+\go(T_{\geq j}))_{i}=(\gr+\go(T_{\geq
j}))_{i+1}.
\end{equation}

Define $T^*$ to be the barred skew tableau of shape $\bgl$
obtained from $T$ by replacing $T_{<j}$ by $s_i(T_{<j})$.  It is
clear that $T^*\in\Bl$, and that $T\mapsto T^*$ defines an
involution on the Bad Guys in $\Bl$. Furthermore, by Lemma
\ref{l.transp_properties}(i), $c_{T^*}(\by)=c_T(\by)$. By Lemma
\ref{l.transp_properties}(ii), $\go((T^*)_{< j})=s_i\go(T_{< j})$.
By (\ref{e.tech}), $\gr+\go((T^*)_{\geq j})=\gr+\go(T_{\geq
j})=s_i(\gr+\go(T_{\geq j}))$. Consequently,
$\gr+\go(T^*)=s_i(\gr+\go(T))$, implying that
$a_{\gr+\go(T^*)}(x)=-a_{\gr+\go(T)}(x)$. Therefore
$c_{T^*}(\by)a_{\gr+\go(T^*)}(x)=-c_T(\by)a_{\gr+\go(T)}(x)$.
Thus the contributions to $\sum c_T(\by)a_{\gr+\go(T)}(x)$ of two
Bad Guys paired under $T\mapsto T^*$ cancel, and the contribution
of any bad guy paired with itself is 0.
\end{proof}

\section{Change of Basis Coefficients}


In this section, let $x\hn=(x_1,\ldots,x_n)$ and
$y=(y_1,y_2,\ldots)$ be two sets of variables. Since
$\{s_\gm(x\hn)\mid \gm\in\cP_n\}$ and $\{s_\gm(x\hn\,|\,y)\mid
\gm\in\cP_n\}$ both form $\bZ[y]$-bases for $\bZ[x,y]^{S_n}$, we
have change of basis formulas
\begin{align}
s_\gl(x\hn\,|\,y)&=\sum_{\gm\in\cP_n} c_{\gl,n}^{\gm}(y)
s_\gm(x\hn),\quad\text{for some } c_{\gl,n}^\gm(y)\in \bZ[y].\label{e.change_basis_1}\\
s_\gl(x\hn)&=\sum_{\gm\in\cP_n} d_{\gl,n}^\gm(y)
s_\gm(x\hn\,|\,y),\quad\text{for some } d_{\gl,n}^\gm(y)\in \bZ[y].\label{e.change_basis_2}
\end{align}
Since (\ref{e.change_basis_1}) is a special case of
(\ref{e.lr_def}), Theorem \ref{t.lr_rule} gives a new rule for the
coefficients $c_{\gl,n}^\gm(y)$ (see also \cite[Section 4]{Mo1},
which gives a different tableau-based rule for
$c_{\gl,n}^\gm(y)$). The main purpose of this section is to give a
new rule for the coefficients $d_{\gl,n}^\gm(y)$; this rule
follows from our rule for $c_{\gl,n}^\gm(y)$ and Proposition
\ref{p.dmn_coeffs_1}(iii) below.

We point out that all formulas in Proposition \ref{p.dmn_coeffs_1}
can be deduced as special cases of formulas involving double
Schubert polynomials appearing in Macdonald \cite[Chapter 6]{Mac3}
and Lascoux \cite[Theorem 10.2.6]{La2}. Our approach is to instead
use a formula by Macdonald \cite{Mac2} for $c_{\gl,n}^\gm(y)$,
which appears as Proposition \ref{p.dmn_coeffs_1}(i) below, in
order to prove Proposition \ref{p.dmn_coeffs_1}(ii) and (iii). All
proofs in this section are replicas or modifications of proofs
appearing in Macdonald \cite[Chapter 1]{Mac1}.

Let $r$ be a nonnegative integer and $p$ a positive integer. The
\textbf{$r$-th elementary and complete symmetric polynomials in
variables} $y_1,\ldots,y_p$, denoted by $e_r(y_{(p)})$ and
$h_r(y_{(p)})$ respectively, are defined by the following
generating functions:
\begin{equation}\label{e.gen_sym_polys}
\begin{split}
E(y_{(p)},t)&=\prod_{i=1}^p (1+y_it)=\sum_{r\geq 0}e_r(y_{(p)})
t^r.\\
H(y_{(p)},t)&=\prod_{i=1}^p (1-y_it)^{-1}=\sum_{r\geq
0}h_r(y_{(p)}) t^r.
\end{split}
\end{equation}
It follows that $e_{0}(y_{(p)})=h_{0}(y_{(p)})=1$, and
$e_{r}(y_{(p)})=0$ for $r>p$. For $r<0$, $e_{r}(y_{(p)})$ and
$h_{r}(y_{(p)})$ are defined to be $0$. For
$\gl=(\gl_1,\ldots,\gl_n)\in\cP_n$, let $m$ be any integer greater
than or equal to $\gl'_1$, the number of columns of $\gl$. Define
$\gl^c=(\gl^c_1,\ldots,\gl^c_m)\in \cP_m$, the
\textbf{complementary partition to $\gl$}, by
$\gl^c_i=n-\gl'_{m+1-i}$, $i=1,\ldots,m$.
\begin{figure}[!h]
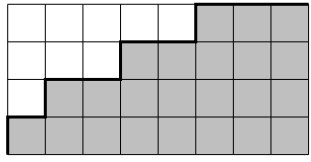

\begin{center}
\psset{unit=.5cm,linewidth=.02} \pspicture(0,0)(8,4)
\psset{linecolor=lightgray}
\psframe*(0,0)(5,1)
\psframe*(1,1)(5,2)
\psframe*(3,2)(5,3)
\psframe*(5,0)(8,4)
\psset{linecolor=black}
\psline(0,0)(8,0)
\psline(0,1)(8,1)
\psline(0,2)(8,2)
\psline(0,3)(8,3)
\psline(0,4)(8,4)
\psline(0,0)(0,4)
\psline(1,0)(1,4)
\psline(2,0)(2,4)
\psline(3,0)(3,4)
\psline(4,0)(4,4)
\psline(5,0)(5,4)
\psline(6,0)(6,4)
\psline(7,0)(7,4)
\psline(8,0)(8,4)
\psset{linewidth=.08}
\psline(0,0)(0,1)(1,1)(1,2)(3,2)(3,3)(5,3)(5,4)(8,4)
\endpspicture
\end{center}
\caption{\label{f.comp_part}
The partitions $\gl=(5,3,1)$ and $\gl^c=(4,4,4,3,3,2,2,1)$, where $n=4$, $m=8$.}
\end{figure}
\begin{prop}\label{p.dmn_coeffs_1}
(i) $\displaystyle
c_{\gl,n}^\gm(y)=\det\left(e_{\gl_i-\gm_j-i+j}(y_{(\gl_i+n-i)})\right)_{1\leq
i,j\leq n}$.

\ni (ii) $\displaystyle
d_{\gl,n}^\gm(y)=\det\left(h_{\gl_i-\gm_j-i+j}((-y)_{(\gm_i+n+1-i)})\right)_{1\leq
i,j\leq n}$.

\ni (iii) $\displaystyle
d_{\gl,n}^\gm(y)=c_{\gm^c,m}^{\gl^c}(-y)$.
\end{prop}
\ni Before  proving this proposition, we prove the following lemma
(see also \cite[Chapter
1]{Mac1}). Let $N\in\bN$, 
$A=\left(e_{i-j}\left(y_{(i)}\right)\right)_{0\leq i,j\leq N-1}$,
and $B=\left(h_{i-j}\left((-y)_{(j+1)}\right)\right)_{0\leq
i,j\leq N-1}$.
\begin{lem}\label{l.AB_inv} $A$ and $B$ are inverse matrices.
\end{lem}
\begin{proof}
Let $q\geq p$. By (\ref{e.gen_sym_polys}),
$E(y_{(q)},t)H((-y)_{(p)},t)$ is a polynomial of degree $q-p$ in
$t$. Also by (\ref{e.gen_sym_polys}),
\begin{equation*}
E(y_{(q)},t)H((-y)_{(p)},t)=\sum_{M\geq 0}\left(\sum_{r+s=M}
e_r(y_{(q)})h_s((-y)_{(p)})\right)t^M.
\end{equation*}
Thus for $q\geq p$,
\begin{equation}\label{e.poly_vanishing}
\sum_{r+s=M} e_r(y_{(q)})h_s((-y)_{(p)})=0,\ \text{if }M>q-p.
\end{equation}

For $0\leq i,k\leq N-1$, consider
$(AB)_{i,k}=\sum_{j=0}^{N-1}
e_{i-j}(y_{(i)})h_{j-k}((-y)_{(k+1)})$.
If $i<k$, then for each $j\in\{0,\ldots,N-1\}$, either $i-j<0$ or
$j-k<0$; thus $(AB)_{i,k}=0$.  If $i>k$, then $(AB)_{i,k}=0$ by
(\ref{e.poly_vanishing}), $M=i-k$.  If $i=k$, then $(AB)_{i,k}=1$.
This completes the proof.
\end{proof}

\begin{proof}[Proof of Proposition \ref{p.dmn_coeffs_1}]
As noted above, (i) is proven in Macdonald \cite{Mac2}. Define
$\cP_{n,m}=\{\gn\in\cP_n\mid\gn'_1\leq m\}$. Let $N=n+m$, and let
$I_{n,N}$ denote the $n$-element subsets of $\{0,\ldots,N-1\}$,
which we always assume are listed in increasing order.  The map
$\pi:\cP_{n,m}\to I_{n,N}$ given by
$\displaystyle\gn=(\gn_i)_{i=1}^n\mapsto
I_\gn=\{\gn_i+n-i\}_{i=n}^1$, is a bijection.

The matrix $A$ is lower triangular with $1$'s along the diagonal,
so $\det(A)=1$. For $I,J\in I_{n,N}$, let $A_{I,J}$ denote the
$n\times n$ submatrix of $A$ with row set $I$ and column set $J$.
By (i), $c_{\gl,n}^\gm(y)=\det(A_{I_\gl,I_\gm})$. This implies the
following interpretation of $c_{\gl,n}^\gm(y)$: letting
$\wedge^nA$ denote the $\binom{N}{n}\times \binom{N}{n}$ matrix
$(\det(A_{I,J}))_{I,J\in I_{n,N}}$, where the rows and columns of
$\wedge^nA$ are ordered by some order on $I_{n,N}$,
$c_{\gl,n}^\gm=(\wedge^nA)_{I_\gl,I_\gm}$ (note that
$(\wedge^nA)_{I,J}$ refers to a single entry of $\wedge^nA$,
whereas $A_{I,J}$ refers to an $n\times n$ submatrix of $A$). By
Lemma \ref{l.AB_inv}, $\wedge^nB=(\wedge^nA)^{-1}$, and thus
$d_{\gl,n}^\gm(y)=(\wedge^nA)^{-1}_{I_{\gl},I_{\gm}}=(\wedge^nB)_{I_{\gl},I_{\gm}}=\det(B_{I_\gl,I_\gm})$.
This proves (ii).

To prove (iii), we give a different expression for
$(\wedge^nA)^{-1}$.  For $I\in I_{n,N}$, define $\gr_I=\#\{j<i\mid
i\in I, j\in I'\}$ and $I'=\{0,\ldots,N-1\}\setminus I\in
I_{m,N}$. The following formula gives the Laplace expansion for
determinants (see, for example, \cite[III, \S8, no. 6]{Bo}): for
$I\in I_{n,N}$, $K\in I_{N-n,N}$,
\begin{equation*}
\sum_{J\in I_{n,N}} (-1)^{\gr(I)+\gr(J)} \det(A_{I,J})\det(A_{K,J'})=
\begin{cases}
\det(A) &\text{if }K=I'\\
0 &\text{if }K\neq I'
\end{cases}.
\end{equation*}
Thus, since $\det(A)=1$, $\wedge^nA$ is invertible, and
\begin{equation*}
(\wedge^nA)^{-1}=\left((-1)^{\gr(I)+\gr(J)}\det(A_{J',I'})\right)_{I,J\in I_{n,N}}.
\end{equation*}
Consequently,
\begin{equation}\label{e.dlmn1}
d_{\gl,n}^\gm(y)=(\wedge^nA)^{-1}_{I_{\gl},I_{\gm}}=(-1)^{\gr(I_\gl)+\gr(I_\gm)}\det(A_{I'_{\gm},I'_{\gl}}).
\end{equation}

For $\gn\in\cP_{n,m}$, consider the following two elementary
properties of $I_\gn$.
\begin{equation*}
1.\ \gr(I_\gn)=|\gn|-n
\qquad\qquad
2.\ I'_\gn=I_{\gn^c}
\end{equation*}
\ni To prove property 1, note that for $I_\gn=\{i_1,\ldots,i_n\}$,
$i_1<\cdots<i_n$, we have $\gr(I_\gn)=(i_1-1)+\cdots+(i_n-n)$ and
$|\gn|=(i_1-0)+\cdots+(i_n-(n-1))$. To prove property 2, partition
the rectangular Young diagram $D$ with $n$ rows and $m$ columns as
$D=\gn\,\dot{\cup}\,\gn^c$ (see Figure \ref{f.comp_part}). Number
the boundary segments, which are darkened in Figure
\ref{f.comp_part}, from $0$ to $m+n-1$. The numbers on the
vertical and horizontal segments are the elements of $I_{\gn}$ and
$I_{\gn^c}$ respectively. This completes the proof of the property
2  (see also \cite[(1.7)]{Mac1}).

Applying these two properties to (\ref{e.dlmn1}),
\begin{equation*}
d_{\gl,n}^\gm(y)
=(-1)^{|\gl|+|\gm|}\det(A_{I_{\gm^c},I_{\gl^c}})
=(-1)^{|\gl|+|\gm|}c_{\gm^c,m}^{\gl^c}(y).
\end{equation*}
Proposition \ref{p.dmn_coeffs_1}(iii) now follows from the fact
that $c_{\gm^c,m}^{\gl^c}(y)$ is a homogeneous polynomial of
degree $|\gm^c|-|\gl^c|=(nm-|\gm|)-(nm-|\gl|)=|\gl|-|\gm|$; thus
$(-1)^{|\gl|+|\gm|}c_{\gm^c,m}^{\gl^c}(y)
=(-1)^{|\gl|-|\gm|}c_{\gm^c,m}^{\gl^c}(y)
=c_{\gm^c,m}^{\gl^c}(-y)$.
\end{proof}

\begin{rem} Theorem \ref{t.lr_rule} combined with Proposition \ref{p.dmn_coeffs_1}
leads to a solution to the problem discussed in the
introduction of expanding $s_{\bgl}\xyb$ in the basis of factorial
Schur functions:
$s_{\bgl}\xyb=\sum_{\gm\in\cP_n}e\blm(\by)s_\gm\xyi$, where
\begin{equation*}
e\blm(\by)=\sum_{\gn\in\cP_n}c_{\bgl}^\gn(\by)\,d_{\gn,n}^\gm(y^{(i)}).
\end{equation*}
Additionally, one can express $c_{\bgl}^\gn(\by)$ in terms of
change of basis coefficients and classical Littlewood-Richardson
coefficients.  For example, for $r=2$,
\begin{equation*}
c_{\bgl}^\gm(\by)=\sum_{\ga,\gb\in\cP_n}c_{\gl^{(1)},n}^\ga(y^{(1)})\,
c_{\gl^{(2)},n}^\gb(y^{(2)})\,c_{\ga,\gb}^\gm,
\end{equation*}
where $c_{\ga,\gb}^\gm\in\bZ$ is the classical
Littlewood-Richardson coefficient.
\end{rem}



\providecommand{\bysame}{\leavevmode\hbox
to3em{\hrulefill}\thinspace}
\providecommand{\MR}{\relax\ifhmode\unskip\space\fi MR }
\providecommand{\MRhref}[2]{%
  \href{http://www.ams.org/mathscinet-getitem?mr=#1}{#2}
} \providecommand{\href}[2]{#2}

\end{document}